\documentclass{amsart}
\usepackage{amssymb, amsthm, amsmath, amsfonts}
\usepackage{graphics}
\usepackage[all]{xy}
\usepackage{hyperref}
\usepackage{enumerate}
\usepackage[mathscr]{eucal}
\usepackage{thmtools}
\usepackage{bbm}
\usepackage{color}

\providecommand{\U}[1]{\protect\rule{.1in}{.1in}}
\def\theenumi{\arabic{enumi}}

\def\theenumii{\alph{enumii}}
\def\p@enumii{\theenumi.}

\def\theenumiii{\arabic{enumiii}}
\def\p@enumiii{(\theenumi)(\theenumii)}

\def\p@enumiv{\p@enumiii.\theenumiii}
\theoremstyle{plain}
\newtheorem{theorem}{Theorem}[section]
\newtheorem{lemma}[theorem]{Lemma}
\newtheorem{proposition}[theorem]{Proposition}

\newtheorem{corollary}[theorem]{Corollary}

\numberwithin{equation}{section}

\theoremstyle{definition}

\newtheorem{definition}[theorem]{Definition}
\newtheorem{example}[theorem]{Example}
\newtheorem{remark}[theorem]{Remark}

\newtheorem{thmab}{Theorem}

\renewenvironment{proof}[1][\proofname]{{\bfseries #1\\}}{\qed}

\DeclareMathOperator{\Hom}{Hom}

\DeclareMathOperator{\Tor}{Tor}

\DeclareMathOperator{\HD}{HD}

\DeclareMathOperator{\td}{td}
\DeclareMathOperator{\CMreg}{CMreg}
\DeclareMathOperator{\coker}{coker}
\DeclareMathOperator{\Ext}{Ext}
\DeclareMathOperator{\Res}{Res}
\DeclareMathOperator{\Ind}{Ind}

\newcommand{\Sn}{\mathfrak{S}}

\newcommand{\C}{{\mathscr{C}}}

\newcommand{\Zn}{\mathbb{Z}}

\newcommand{\N}{\mathbb{N}}

\newcommand{\FI}{{\mathscr{FI}}}
\newcommand{\mi}{\mathfrak{m}}
\newcommand{\filcom}{\mathcal{C}^\bullet}
\newcommand{\fiHom}{\mathscr{H}om}
\newcommand{\fiExt}{{\mathscr{E}xt}}
\newcommand{\arXiv}[1]{\href{http://arxiv.org/abs/#1}{\nolinkurl{arXiv:#1}}}
\newcommand{\arXivV}[2]{\href{http://arxiv.org/abs/#1}{\nolinkurl{arXiv:#1v#2}}}

\title{Local cohomology and the multi-graded regularity of $\FI^m$-modules}

\author{Liping Li and Eric Ramos}
\address{HPCSIP (Ministry of Education), College of Mathematics and Computer Science, Hunan Normal University, Changsha, Hunan 410081, China.}
\email{lipingli@hunnu.edu.cn}
\address{Department of Mathematics, University of Michigan - Ann Arbor.}
\email{eramos@math.wisc.edu}

\thanks{L. Li is supported by the National Natural Science Foundation of China 11771135 and the Start-Up Funds of Hunan Normal University 830122-0037. E. Ramos was supported by NSF grant DMS-1704811.}

\begin{document}

\begin{abstract}
We develop a local cohomology theory for $\FI^m$-modules, and show that it in many ways mimics the classical theory for multi-graded modules over a polynomial ring. In particular, we define an invariant of $\FI^m$-modules using this local cohomology theory which closely resembles an invariant of multi-graded modules over Cox rings defined by Maclagan and Smith. It is then shown that this invariant behaves almost identically to the invariant of Maclagan and Smith.
\end{abstract}

\maketitle

\section{Introduction}

There has been a recent boom in the literature of what one might call the representation theory of infinite combinatorial categories. If $k$ is a commutative ring, and $\C$ is a small category, then a \emph{representation of $\C$}, or a \emph{$\C$-module}, is a functor from $\C$ to the category of $k$-modules. One particular category whose representation theory has garnered a tremendous amount of interest is $\FI$, the category of finite sets and injections. This is largely due to the seminal work of Church, Ellenberg, and Farb \cite{CEF}, who showed that $\FI$ has deep connections to representation stability theory \cite{CF}. This has lead to an outpouring of results applying $\FI$-modules to a variety of different subjects (see \cite{CE,CEFN,CMNR,N2} and the references therein).\\

While much of the original interest in studying $\FI$-modules was based on newly discovered applications to fields such as topology and number theory, there has also been a push to understand these objects from a more traditional representation theoretic perspective (see \cite{G,GL,LY,R,SS} for a small sampling of such results). Previous work of the authors \cite{LR} was focused on how the representation theory of $\FI$-modules closely resembled the very classical study of graded modules over polynomial rings. In particular, it was shown that one could make sense of a \emph{local cohomology} theory for $\FI$-modules, extending previous work of Sam and Snowden \cite{SS2}, and that this theory closely resembled the local cohomology theory of graded modules over polynomial rings.\\

Using $\FI$-modules as a launching point, it has recently become interesting to study representations of other categories which arise naturally in various contexts. Just as with $\FI$-modules, these other categories are now being considered from representation theoretic points of view (see \cite{G,W,LY2,N}). In this paper, we will be considering the representation theory of the category $\FI^m$, the $m$-fold product of $\FI$ with itself. The primary objective of this paper is to show that while the representation theory of $\FI$ closely resembled the study of graded modules over a polynomial ring, $\FI^m$-modules more closely resemble \emph{multi-graded} modules over polynomial rings.\\

Let $V$ denote an $\FI^m$-module. For a given vector $\mathbf{n} = (n_1,\ldots,n_m) \in \N^m$, we will often write $V_\mathbf{n}$ for the evaluation of $V$ at $[n_1] \times [n_2] \times \ldots \times [n_m]$. The \emph{transition maps} of $V$ are the maps induced by the non-invertible morphisms of $\FI^m$. A \emph{$B$-torsion element} of $V_{\mathbf{n}}$ is any element $v \in V_{\mathbf{n}}$ which is in the kernel of some transition map originating from $V_{\mathbf{n}}$. The $B$-torsion functor is that which sends an $\FI^m$-module to its maximal $B$-torsion submodule, while the \emph{local cohomology functors} are the derived functors of this $B$-torsion functor. We denote these functors by $H^i_B(\bullet)$.\\

Once one has established a local cohomology theory, one of the first questions one might ask is whether the local cohomology modules of a given module are in any way related to syzygies of that module. That is, are the local cohomology modules in any way related to the \emph{regularity} of the given module. Such a relation was first noted in the case of $\FI$-modules in \cite{LR}, and later expanded upon by Nagpal, Sam, and Snowden \cite{NSS}. For a variety of reasons, answering this fundamental question becomes considerably more subtle when $m > 1$. Our approach is designed to closely resemble the notion of regularity first introduced by Maclagan and Smith \cite{MS}. In particular, we treat regularity as a set, as opposed to a single value.\\

Let $V$ be an $\FI^m$-module. Then the \emph{positive Castelnuovo-Mumford regularity} of $V$ is the set of vectors $\mathbf{r} \in \Zn^m$ satisfying:
\[
H^i_B(V)_{\mathbf{n}} = 0 \text{ whenever } \mathbf{n} \in \cup_{(a_1,\ldots,a_m)}\cup_j(\mathbf{r}-(a_1,\ldots,a_m) + e_j + \N^m).
\]
Here, $i \geq 0$, the outer union is over vectors $(a_1,\ldots,a_m) \in \N^m$ with $\sum_j a_j = i$, and $e_j$ is the usual standard basis vector. Note that this definition is almost identical to the definition of \emph{multi-graded regularity} given by Maclagan and Smith for multi-graded modules over Cox rings in \cite{MS}. The main theorem of this paper is also almost identical to the main theorem of that work.\\

We denote the positive Castelnuovo-Mumford regularity of $V$ by $\CMreg_+(V) \subseteq \Zn^m$.\\

\begin{thmab}\label{reglc}
Let $V$ be a finitely generated $\FI^m$-module (see Definition \ref{fg}) over a Noetherian ring $k$. Then for all vectors $\mathbf{c}$ with positive integral coordinates, and for all $\mathbf{r} \in \CMreg_+(V) \cap \HD_0(V)$ (see Definition \ref{homdegree}) there exists a complex
\[
\ldots \rightarrow F_1 \rightarrow F_0 \rightarrow V \rightarrow 0
\]
such that:
\begin{enumerate}
\item $F_i$ is an \emph{induced module} (see Definition \ref{induced}) generated in degree $\mathbf{r} + i\cdot \mathbf{c}$;
\item the homology of this complex is $B$-torsion.\\
\end{enumerate}
\end{thmab}

\begin{remark}
Induced modules are analogous to free modules, and may usually be thought of as such to develop intuition. For instance, it is a fact that all projective $\FI^m$-modules are induced. Despite this, these modules are also local cohomology acyclic (see Theorem \ref{acyclicclass}). As a consequence, it is not possible for local cohomology to detect the generating degrees of modules. This justifies why the above theorem explicitly picks vectors from $\HD_0(V)$. One may therefore think of the above theorem as stating that while local cohomology cannot detect generating degree, it can detect all higher syzygies up to $B$-torsion.\\
\end{remark}

This paper is structured as follows. We begin by recalling:\\
\begin{itemize}
\item basic definitions from the study of $\FI^m$-modules (Section 2.1);
\item the shift functors and their relationship to torsion (Section 2.2);
\item $\FI$-module homology and homological degrees (Section 2.3).\\
\end{itemize}

Following these background sections, we turn our attention to developing the local cohomology theory of $\FI^m$-modules (Section 3). While this work is largely novel, it also closely resembles the theory for $\FI$-modules and should not be considered the major contribution of this work. We conclude the paper by discussing Castelnuovo-Mumford regularity, and performing the proof of Theorem \ref{reglc} (Section 4).\\

\section{Background}

\subsection{Basic definitions}

\begin{definition}\label{fg}
The category $\FI$ is that whose objects are the finite sets $[n] := \{1,\ldots,n\}$ and whose morphisms are injections. For a fixed positive integer $m$, we define $\FI^m$ to be the $m$-fold categorical product of $\FI$ with itself. Namely, the objects of $\FI^m$ are tuples $[n_1] \times \ldots \times [n_m]$, while the morphisms are given by products $f_1 \times \ldots f_m$ where for each $j$, $f_j:[n_j] \rightarrow [r_j]$ is an injection. Given an integral vector $\mathbf{n} = (n_1,\ldots,n_m) \in \N^m$, where $\N$ is the set of nonnegative integers and $[0] = \emptyset$ by convention, we will write $[\mathbf{n}]$ as a shorthand for $[n_1] \times \ldots \times [n_m]$. Similarly, we write $\Sn_{\mathbf{n}}$ for the product of symmetric groups $\Sn_{n_1} \times \ldots \times \Sn_{n_m}$\\

For a fixed Noetherian commutative ring $k$, an \textbf{$\FI^m$-module over $k$} is a covariant functor from $\FI^m$ to the category of $k$-modules. often times, when the ring $k$ is understood, we will refer to $\FI^m$-modules over $k$ as simply being $\FI^m$-modules. Given an $\FI^m$-module $V$, and a vector $\mathbf{n} = (n_1,\ldots,n_m) \in \N^m$, we write
\[
V_{\mathbf{n}} = V_{n_1,\ldots,n_m} := V([n_1] \times \ldots \times [n_m])
\]
A \textbf{transition map} of $V$ is any map induced by a map $f_1 \times \ldots \times f_m$ such that at least one $f_i$ is not invertible.\\

The category of $\FI^m$-modules is abelian, with the usual abelian operations defined pointwise. This allows us to formulate obvious definitions for things such as \textbf{submodules} and \textbf{quotient modules}. We say that an $\FI^m$-module is \textbf{finitely generated} if there is a finite collection of elements
\[
\{v_j\} \subseteq \sqcup_{\mathbf{n}}V_\mathbf{n}
\]
such that no proper submodule of $V$ contains all of the $v_j$.\\
\end{definition}

It is a fact that the category of $\FI^m$-modules over any Noetherian commutative ring $k$ is locally Noetherian. That is, submodules of finitely generated $\FI^m$-modules are once again finitely generated. This was proven for general $\FI^m$-modules by Sam and Snowden in \cite{SS}, as well as by Gadish in \cite{G} when $k$ is a field of characteristic 0. The specific case of $\FI$-modules was proven earlier by Snowden \cite{S}, as well as by Church, Ellenberg, Farb and Nagpal \cite{CEF,CEFN}. We will be using this Noetherian property all throughout the current work.\\

\emph{For the remainder of this paper, $m$ will be reserved to denote the positive integer in $\FI^m$. We also fix a commutative Noetherian ring $k$.}\\

\begin{definition}\label{induced}
Fix an integral vector $\mathbf{r} = (r_1,\ldots,r_m) \in \N^m$, as well as a $k[\Sn_{\mathbf{r}}]$-module, $W$. Then we define the \textbf{induced $\FI^m$-module on $W$} by the assignment
\[
M(W)_{\mathbf{n}} = k[\Hom_{\FI^m}(\mathbf{r},\mathbf{n})] \otimes_{k[\Sn_{\mathbf{r}}]} W
\]
where $k[\Hom_{\FI^m}(\mathbf{r},\mathbf{n})]$ is the free $k$-module with basis indexed by the set $\Hom_{\FI^m}(\mathbf{r},\mathbf{n})$. The $\FI^m$-morphisms act on $M(W)$ by composition on the left factor. In the specific cases wherein $W = k[\Sn_{\mathbf{r}}]$ is the regular representation, we write $M(W) = M(\mathbf{r})$.\\

An $\FI^m$-module is said to be \textbf{semi-induced} if it admits a finite filtration whose every co-factor is isomorphic to an induced module.\\
\end{definition}

Induced and semi-induced modules play a pivotal role in the study of $\FI^m$-modules. For instance, we have the following.\\

\begin{proposition}\label{inducedproperties}
Let $W$ denote a $k[\Sn_{\mathbf{r}}]$-module for some fixed integral vector $\mathbf{r} \in \N^m$, and let $V$ be an $\FI^m$-module. Then:
\begin{enumerate}
\item $\Hom _{\FI^m}(M(W),V) \cong \Hom_{k[\Sn_{\mathbf{r}}]}(W,V_{\mathbf{r}})$;
\item $V$ is finitely generated if and only if there is a surjective map of $\FI^m$-modules
\[
X \rightarrow V \rightarrow 0
\]
where $X$ is some finitely generated semi-induced $\FI^m$-module;
\item if $V$ is finitely generated and projective, then it is semi-induced.
\item if $V$ is finitely generated, then it is acyclic with respect to the \textbf{$\FI^m$ homology functors} if and only if it is semi-induced (see Definition \ref{homdegree}).\\
\end{enumerate}
\end{proposition}

Proofs of the above facts can be found in \cite{LY2}. Each of the parts of Theorem \ref{inducedproperties} illustrate the significance of semi-induced modules when studying right exact functors on the category of finitely generated $\FI^m$-modules. One recent insight, seen in the authors' precurser work \cite{LR} as well as in work of Nagpal \cite{N} and Church, Miller, Nagpal and Reinhold \cite{CMNR}, is that semi-induced modules are also fundamental to the study of left exact functors on the category of finitely generated $\FI^m$-modules. We will later see this when we consider the various local cohomology functors (see Definition \ref{localcohomology}).\\

To conclude this section, we record various notation that will be used throughout the paper.\\

\begin{definition}
For each $1 \leq j \leq m$ we write $e_j$ to denote the usual standard basis vector of $\N^m$. If $\mathbf{r} = (r_1,\ldots,r_m) \in \N^m$, then we use $|\mathbf{r}|$ to denote the sum of its coordinates. Given two vectors $\mathbf{r}, \mathbf{n} \in \N^m$, we say that $\mathbf{r}$ is \textbf{at least one positive step} from $\mathbf{n}$ if $\mathbf{r} \in \cup_j (\mathbf{n} + e_j + \N^m)$.\\
\end{definition}

\subsection{The shift functors and torsion}

In this section we briefly review the shift functors and the various notions of torsion. The shift functors were originally defined by Church, Ellenberg, Farb, and Nagpal for $\FI$-modules in \cite{CEF} and \cite{CEFN}. Their many properties were later discovered and used throughout the literature \cite{C,CE,L,LY,LR,N,R}. The significance of torsion in the study of $\FI$-modules was considered by Sam and Snowden in \cite{SS2}, as well as by the authors in \cite{LR}.\\

In \cite{LY2}, The first author and Yu consider analogs of the shift functors for $\FI^m$-modules. The work in this section traces itself back to that paper.\\

\begin{definition}
Let $1 \leq j \leq m$ be an integer. Then we define the functor $\iota_j:\FI^m \rightarrow \FI^m$ by
\[
\iota_j([\mathbf{n}]) = [n_1] \times \ldots \times [n_{j-1}] \times [n_j + 1] \times [n_{j+1}] \times \ldots \times [n_m]
\]
and for $f_1 \times \ldots \times f_m: [\mathbf{n}] \rightarrow [\mathbf{n}]$,
\[
\iota_j(f_1 \times \ldots \times f_m)(x_1,\ldots,x_m) = \begin{cases} (f_1(x_1),\ldots,f_m(x_m)) &\text{if $x_j \leq n_j$}\\ (f_1(x_1),\ldots,f_{j-1}(x_{j-1}),r_j+1,f_{j+1}(x_{j+1}),\ldots,f_m(x_m)) &\text{otherwise.}\end{cases}
\]
The \textbf{$j$-th shift functor} $\Sigma_j$ is defined on $\FI^m$-modules
\[
\Sigma_j V := V \circ \iota_j
\]
If $\mathbf{a} = (a_1,\ldots,a_m) \in \N^m$, then we write $\Sigma_{\mathbf{a}}$ to denote the composition
\[
\Sigma_{\mathbf{a}}V := (\Sigma_1^{a_1} \circ \ldots \circ \Sigma_m^{a_m})V.
\]
\end{definition}

Perhaps one of the most significant properties of the shift functor is how it interacts with induced and semi-induced modules. The following theorem was first realized by Nagpal in the case of $\FI$-modules \cite{N2}. It was shown for general $\FI^m$-modules by the first author and Yu in \cite{LY2}.\\

\begin{theorem}\label{shiftthm}
Let $W$ denote a $k[\Sn_{\mathbf{r}}]$-module, and let $1 \leq j \leq m$ denote a fixed integer. Then:
\begin{enumerate}
\item $\Sigma_j M(W) \cong M(\Res^{\Sn_{\mathbf{r}}}_{\Sn_{\mathbf{r} - e_j}} W) \oplus M(W)$. In particular, if $V$ is a finitely generated $\FI^m$-module, then the same is true of $\Sigma_j W$.
\item The standard maps $[\mathbf{n}] \rightarrow [\mathbf{n} + e_j]$ induce, for any $\FI^m$-module $V$, a map $V \rightarrow \Sigma_j V$. If $V = M(W)$ is induced, then this map is a split injection.
\item $\HD_0(\Sigma_j V) \supseteq \HD_0(V)$ and $\HD_0(\coker(V \rightarrow \Sigma_j V)) + e_j \supseteq \HD_0(V)$ (see Definition \ref{homdegree})
\item If $V$ is a finitely generated $\FI^m$-module, then there exists a non-negative integer $N(V)$ such that for all integral vectors $\mathbf{a} = (a_1,\ldots,a_m) \in \N^m$ with $a_i \geq N(V)$ for each $i$, $\Sigma_\mathbf{a}V$ is semi-induced.
\item The functor $\Sigma_j$ preserves injective objects. That is, if $I$ is an injective $\FI^m$-module, then the same is true of $\Sigma_jI$.
\end{enumerate}
\end{theorem}

\begin{proof}
The only part of this theorem that does not appear in \cite{LY2} is the last part. We will prove the following stronger statement: The functor $\Sigma_j$ admits an exact left adjoint, $\Ind_j$.\\

In the case $m = 1$, the existence of $\Ind_j$ is shown in \cite{LR}. This construction can be applied to the general situation in similar fashion. This functor is exact, as it is given on points by
\[
(\Ind_j V)_{\mathbf{n}} = \Ind_{\Sn_{\mathbf{n}-e_j}}^{\Sn_\mathbf{n}} V_{\mathbf{n}-e_j}
\]
and note that $k\FI^m(\mathbf{n}-e_j, \mathbf{n})$ is a free right $k\Sn_{\mathbf{n}-e_j}$-module.
\end{proof}

The third part of Theorem \ref{shiftthm} is particularly significant. We will later see that it is the key piece needed to prove certain finiteness statements for local cohomology, along with countless other facts. See \cite{LR,LY2,N,N2} for some of the countless applications of this theorem in a variety of contexts.\\

For the purposes of this paper, the relevance of the shift functors is mostly in their connection to torsion.\\

\begin{definition}
Let $V$ be an $\FI^m$-module. An element $v \in V_{\mathbf{n}}$ is said to be \textbf{$B$-torsion} if it is in the kernel of some transition map $V_{\mathbf{n}} \rightarrow V_{\mathbf{r}}$. We say that $V$ is \textbf{$B$-torsion} if all of its elements are $B$-torsion.\\

If $1 \leq j \leq m$, then we say an element $v \in V_{\mathbf{n}}$ is \textbf{$B$-torsion in direction $j$} if it is in the kernel to all transition maps into $V_{\mathbf{n} + e_j}$. Note that this is equivalent to being in the kernel of the map $V \rightarrow \Sigma_j V$ of Theorem \ref{shiftthm}.\\
\end{definition}

\begin{remark}
The notion of $B$-torsion is borrowed from algebraic geometry and the study of Cox rings. In that setting, $B$ is usually used to denote the geometrically irrelevant ideal and $B$-torsion is the basis for local cohomology. The same will be true in this context.\\
\end{remark}

The following collects most of the important facts about $B$-torsion in $\FI^m$-modules. Proofs can be found throughout \cite{LY2}.\\

\begin{proposition}\label{torsionprop}
Let $V$ be an $\FI^m$-module. Then:
\begin{enumerate}
\item there is an exact sequence
\[
0 \rightarrow V_T \rightarrow V \rightarrow V_F \rightarrow 0
\]
where $V_T$ is $B$-torsion, and $V_F$ is $B$-torsion-free;
\item if $V$ is $B$-torsion and finitely generated, then there exists a positive integer, $\td(V)$, such that $V_{\mathbf{r}} = 0$ for all $\mathbf{r} = (r_1,\ldots,r_m)$ with $r_i \geq \td(V)$ for all $i$.;
\item $V$ is $B$-torsion-free if and only if the map $V \rightarrow \Sigma_{(1,1,\ldots,1)}V$ from the second part of Theorem \ref{shiftthm} is injective.\\
\end{enumerate}
\end{proposition}

\subsection{$\FI^m$-module Homology}

In this section we recall the definition of $\FI^m$-module homology. The homology functors were introduced by Church, Ellenberg, Farb, and Nagpal in the case of $\FI$-modules \cite{CEFN}. Following this, they were deeply explored by a large collection of authors \cite{C,CE,GL,LR,LY,R}. As with previous sections, the study of $\FI^m$-module homology originates with the work of the first author and Yu \cite{LY2}.\\

\begin{definition}\label{homdegree}
Let $V$ be an $\FI^m$-module. The \textbf{0-th homology of $V$} is the $\FI^m$-module defined on points by
\[
H_0(V)_{\mathbf{n}} = V_{\mathbf{n}}/ V_{< \mathbf{n}}
\]
where $V_{< \mathbf{n}}$ is the submodule of $V_{\mathbf{n}}$ generated by elements which are in the image of some transition map. The functor $V \mapsto H_0(V)$ is right exact, and we denote its derived functors by $H_i(\bullet)$. These functors are the \textbf{$\FI^m$ homology functors}.\\

Given a finitely generated $\FI^m$-module $V$, and an integer $i \geq 0$, we define the \textbf{$i$-th homological degree} of $V$ to be the set
\[
\HD_i(V) := \{\mathbf{r} \in \Zn^m \mid H_i(V)_{\mathbf{n}} = 0 \text{ whenever } \mathbf{n} \in \cup_j (\mathbf{r} + e_j + \N^m).\}
\]
\end{definition}

\begin{remark}
For minor technical reasons (see the Remark \ref{technicalissue}), we have to consider $\HD_i(V)$ as being a subset of $\Zn^m$ instead of $\N^m$. Note that for this to make sense we simply assert that $V_{\mathbf{n}} = 0$ whenever $\mathbf{n} \in \Zn^m - \N^m$.\\
\end{remark}

Before we continue, we hope to develop some intuition for $\FI^m$ homology which is rooted in more traditional algebra. Let $k\FI^m$ denote the ring whose additive group is given by
\[
k\FI^m = \bigoplus_{\mathbf{n},\mathbf{r}} k[\Hom_{\FI^m}(\mathbf{n},\mathbf{r})]
\]
where $k[\Hom_{\FI^m}(\mathbf{n},\mathbf{r})]$ is the free $k$-module with basis indexed by the set $\Hom_{\FI^m}(\mathbf{n},\mathbf{r})$. Multiplication in this ring is defined on basis vectors by
\[
(f_1 \times \ldots \times f_m) \cdot (g_1 \times \ldots \times g_m) = \begin{cases} f_1 \circ g_1 \times \ldots \times f_m \circ g_m &\text{if this makes sense}\\ 0 &\text{otherwise.}\end{cases}
\]
Then an $\FI^m$-module may be viewed as a left $k\FI^m$-module $V$ which admits a direct sum decomposition
\[
V = \bigoplus_{\mathbf{n}} id_\mathbf{n} \cdot V
\]
where $id_{\mathbf{n}}$ is the identity function on $[\mathbf{n}]$. Then there is an ideal in $k\FI^m$, denoted by $\mi$, which is spanned by basis vectors which are not bijections. In this language, the homology functors are precisely given by
\[
H_i(V) = \Tor_i(k\FI^m/\mi, V)
\]
Thinking of $k\FI^m$ as being analogous to a multi-graded polynomial ring, and $\mi$ as being the unique maximal homogeneous ideal, the homology can be thought of as capturing the generating degrees of syzygies of $V$. In particular, we have the following.\\

\begin{proposition}
Let $V$ be a finitely generated $\FI^m$-module. Then there exists an exact sequence
\[
\cdots \rightarrow F^{(1)} \rightarrow F^{(0)} \rightarrow V \rightarrow 0
\]
such that:
\begin{enumerate}
\item $F^{(i)}$ is semi-induced for each $i \geq 0$;
\item $HD_0(F^{(0)}) = HD_0(V)$;
\item $HD_i(V) \supseteq HD_0(F^{(i)})$ for $i \geq 1$.\\
\end{enumerate}
\end{proposition}

One of the main goals of this paper is to determine some kind of connection between the homology of an $\FI$-module and its $B$-torsion. For $\FI$-modules, the relationship between homology and $B$-torsion has been made totally precise, and we will visit this connection later when we define local cohomology. For general $\FI^m$-modules the connection becomes a bit more tenuous. To get an idea for why this is, recall from above that homology is defined to be $H_i(V) = \Tor_i(k\FI^m/\mi,V)$. On the other hand, $B$-torsion elements in this language are elements $v \in V$ for which $B^N v = 0$ for all $N \gg 0$, where $B$ is the ideal of $k\FI^m$ generated by maps $f_1 \times \ldots \times f_m:[\mathbf{n}] \rightarrow [\mathbf{r}]$ for which $r_i > n_i$ for each $i$. This is similar to the setting of modules over multi-graded polynomial rings, wherein one must distinguish the ideal generated by the variables of the polynomial ring from the ideal which is geometrically irrelevant. In the case $m = 1$, $\mi$ and $B$ are the same ideal, and the relationship between homology and torsion becomes much more accessible.\\

For now, we state the following theorem, which shows that every finitely generated $\FI^m$-module can be approximated by $B$-torsion modules and homology acyclic modules. This theorem was famously proven by Nagpal \cite{N2} for $\FI$-modules, and was more recently generalized to $\FI^m$-modules by the first author and Yu in \cite{LY2}.\\

\begin{theorem}\label{cechcomplex}
Let $V$ be a finitely generated $\FI^m$-module. There there exists a bounded complex of finitely generated $\FI^m$-modules
\[
\filcom V: 0 \rightarrow V \rightarrow F^{(0)} \rightarrow \ldots \rightarrow F^{(l)} \rightarrow 0
\]
such that:
\begin{enumerate}
\item $H^i(\filcom V)$ is $B$-torsion for all $i$;
\item $F^{(i)}$ is semi-induced for all $i$.\\
\end{enumerate}
\end{theorem}

\begin{remark}
Note that the assignment $V \mapsto \filcom V$ is \emph{not} functorial.\\
\end{remark}

Just as was the case for $\FI$-modules \cite{LR}, we will find that the complex $\filcom V$ is actually computing the local cohomology modules of $V$ (Theorem \ref{cechcomplexlc}).\\

\section{Local cohomology}

\subsection{Definition and first properties}

\begin{definition}\label{localcohomology}
Let $V$ be an $\FI^m$-module. We define the \textbf{$B$-torsion functor} by the assignment
\[
H^0_B(V) := V_T
\]
where $V_T$ is the $B$-torsion module guaranteed by the first part of Proposition \ref{torsionprop}. The $B$-torsion functor is left exact, and therefore admits right derived functors, which we denote $H^i_B(V)$. The collection of functors $H^i_B(V)$ are known as the \textbf{local cohomology functors}.\\
\end{definition}

\begin{remark}
The category of $\FI^m$-modules is Grothendieck by virtue of the fact that it is a functor category from a small category into a module category. From this it follows that it has sufficiently many injective modules, and one may make sense of the local cohomology functors. However, it is rarely the case that these injective modules are also finitely generated. As a consequence, it is a priori unclear whether the modules $H^i_B(V)$ are finitely generated, even assuming that $V$ is. We will later prove that it is actually always the case that $H^i_B(V)$ is finitely generated assuming that $V$ is (see Corollary \ref{finite}).\\
\end{remark}

It will be useful to us to have an alternative definition of local cohomology in terms of a directed limit of $\Ext$ functors.\\

\begin{definition}
Let $\mathbf{a} = (a_1,\ldots,a_m) \in \N^m$, and let $N > 0$ be a fixed positive integer. We define $B^N\cdot M(\mathbf{a})$ to be the submodule of $M(\mathbf{a})$ generated by $M(\mathbf{a})_{\mathbf{a} + N\cdot(1,\ldots,1)}$.\\
\end{definition}

\begin{lemma}
Let $V$ be an $\FI^m$ module. Then for each $\mathbf{a} = (a_1,\ldots,a_m) \in \N^m$, and $N > 0$ the $k[\Sn_\mathbf{a}]$-module $\Hom(M(\mathbf{a})/B^N \cdot M(\mathbf{a}), V)$ is isomorphic to the submodule of $V_{\mathbf{a}}$ of elements which are in the kernel of some transition map $f_1 \times \ldots \times f_m: [\mathbf{a}] \rightarrow [\mathbf{n}]$ with $\max_i{n_i - a_i} \leq N$.\\
\end{lemma}

\begin{proof}
By the first part of Theorem \ref{inducedproperties}, a map $M(\mathbf{a})/B^N \cdot M(\mathbf{a}) \rightarrow V$ is uniquely determined by a choice of element of $V_{\mathbf{a}}$ which is in the kernel of every transition map into $V_{\mathbf{a}+N \cdot (1,\ldots,1)}$. It remains to show that this condition is equivalent to being in the kernel of at least one transition map with the stated condition on its image.\\

One direction is clear. Conversely, assume that $v \in V_{\mathbf{a}}$ is in the kernel of some transition map $f_1 \times \ldots \times f_m: [\mathbf{a}] \rightarrow [\mathbf{n}]$ with $\max_i{n_i - a_i} \leq N$. Then we may compose this function with a map $[\mathbf{n}] \rightarrow [\mathbf{a}+N \cdot (1,\ldots,1)]$ and find that $v$ is in the kernel of some transition map into $V_{\mathbf{a} + N \cdot (1,\ldots,1)}$. The action of the symmetric group allows us to now conclude $v$ is in the kernel of every transition map of the required form.\\
\end{proof}

\begin{lemma}
Let $V$ be an $\FI^m$-module, and fix $N > 0$. Then there exists an $\FI^m$-module $\fiHom(k\FI^m/B^N,V)$ for which
\[
\fiHom(k\FI^m/B^N,V)_{\mathbf{a}} = \Hom(M(\mathbf{a})/B^N \cdot M(\mathbf{a}), V),
\]
such that the inclusions $\Hom(M(\mathbf{a})/B^N \cdot M(\mathbf{a}), V) \hookrightarrow V_{\mathbf{a}}$ induce an inclusion $\fiHom(k\FI^m/B^N,V) \hookrightarrow V$.\\
\end{lemma}

\begin{proof}
Let $f_1 \times \ldots \times f_m: [\mathbf{a}] \rightarrow [\mathbf{n}]$ be a morphism in $\FI^m$, and let $\phi \in \Hom(M(\mathbf{a})/B^N \cdot M(\mathbf{a}), V)$. Then we define the image of $\phi$ under the transition map induced by $f_1 \times \ldots \times f_m$ to be the morphism $\psi \in \Hom(M(\mathbf{n})/B^N \cdot M(\mathbf{n}), V)$ uniquely determined by the assignment $\psi(id_{\mathbf{n}}) = V(f_1 \times \ldots \times f_m)(\phi(id_{\mathbf{a}})).$\\

The second part of the statement follows immediately from the construction.\\
\end{proof}

These two lemmas lead into the following very natural alternative definition of local cohomology.\\

\begin{definition}
For each fixed $N > 0$, the functor $V \mapsto \fiHom(k\FI^m/B^N,V)$ is left exact. We denote the right derived functors of this functor by
\[
\fiExt^i(k\FI^m/B^N,\bullet).
\]
We note that for each $\mathbf{a} \in \N^m$ one has
\[
\fiExt^i(k\FI^m/B^N,V)_{\mathbf{a}} = \Ext^i(M(\mathbf{a})/B^N \cdot M(\mathbf{a}), V).
\]
\end{definition}

\begin{theorem}\label{altdef}
There is an isomorphism of functors
\[
\lim_{\rightarrow} \fiHom(k\FI^m/B^N,\bullet) \cong H^0_B(\bullet).
\]
inducing an isomorphism of derived functors
\[
\lim_{\rightarrow} \fiExt^i(k\FI^m/B^N,\bullet) \cong H^i_B(\bullet)
\]
for each $i \geq 0$.\\
\end{theorem}

\begin{proof}
Follows from the previous lemmas and standard homological arguments.\\
\end{proof}

\subsection{Acyclic modules}

The purpose of this section is to classify the modules which are acyclic with respect to the local cohomology functors. As a consequence of this classification, we will be able to prove many different finiteness statements about local cohomology modules.\\

To begin, we prove the previously promised statement that semi-induced modules are local cohomology acyclic. The following argument is based on a similar argument of Nagpal appearing in \cite{N}.\\

\begin{proposition}
If $V$ is a semi-induced $\FI^m$-module, then $H^i_B(V) = 0$ for all $i \geq 0$.\\
\end{proposition}

\begin{proof}
It suffices to prove the statement for $V = M(W)$, where $W$ is some $k[\Sn_{\mathbf{r}}]$-module. Fix $1 \leq j \leq m$, an recall from Theorem \ref{shiftthm} that there is a split injection
\[
V \rightarrow \Sigma_jV.
\]
On the other hand, there are natural isomorphisms
\[
H^0_B(\Sigma_j V) \cong \Sigma_j H^0_B(V),
\]
which induce isomorphisms
\[
H^i_B(\Sigma_j V) \cong \Sigma_j H^i_B(V).
\]
Indeed, this follows from the fact that $\Sigma_j$ is exact, and that it preserves injective objects by the last part of Theorem \ref{shiftthm}. It then follows that the map
\[
H^i_B(V) \rightarrow H^i_B(\Sigma_j V) \cong \Sigma_j H^i_B(V)
\]
is injective. In particular, the module $H^i_B(V)$ has no $B$-torsion in direction $j$. Because $j$ is arbitrary, it follows that $H^i_B(V)$ is actually $B$-torsion-free. The only way this is possible is if $H^i_B(V) = 0$ when $i \geq 0$, as desired.\\
\end{proof}

The second class of modules whose positive local cohomology groups vanish are $B$-torsion-modules themselves. To accomplish this, we use the alternative definition of local cohomology provided by Theorem \ref{altdef}.\\

\begin{proposition}\label{tortrivial}
Let $V$ be a finitely generated $B$-torsion module. Then $H^i_B(V) = 0$ for all $i \geq 1$.\\
\end{proposition}

\begin{proof}
Fix $\mathbf{a} \in \N^m$, and $i \geq 1$. It will suffice to prove that for all $N \gg 0$,
\[
\Ext^i(M(\mathbf{a})/B^N \cdot M(\mathbf{a}),V) = 0
\]
Choose $N$ sufficiently large so that $\Hom (B^N\cdot M(\mathbf{a})_{\mathbf{n}}, V) = 0$ whenever $V_{\mathbf{n}} \neq 0$. Note that for such an $N$ to exist, it is critical that $V$ be finitely generated. We may apply $\Hom(\bullet,V)$ to the exact sequence
\[
0 \rightarrow B^N\cdot M(\mathbf{a}) \rightarrow M(\mathbf{a}) \rightarrow M(\mathbf{a})/B^N\cdot M(\mathbf{a}) \rightarrow 0
\]
to conclude the existence of the following exact sequences,
\begin{eqnarray}
&\Hom(B^N\cdot M(\mathbf{a}),V) \rightarrow \Ext^1(M(\mathbf{a})/B^N\cdot M(\mathbf{a}),V) \rightarrow 0&\\
&0 \rightarrow \Ext^{i-1}(B^N\cdot M(\mathbf{a}),V) \rightarrow  \Ext^i(M(\mathbf{a})/B^N\cdot M(\mathbf{a}),V) \rightarrow 0& \hspace{1cm} (i \geq 2)
\end{eqnarray}
The means by which $N$ was chosen implies that $\Hom(B^N\cdot M(\mathbf{a}),V) = 0$, whence $\Ext^1(M(\mathbf{a})/B^N\cdot M(\mathbf{a}),V) = 0$. For $i \geq 2$, the above tells us that it suffices to compute $\Ext^{i-1}(B^N\cdot M(\mathbf{a}),V)$. These modules are subquotients of modules of the form $\Hom(F,V)$, where $F$ is a projective module which is generated in degrees which are at least one positive step from $\mathbf{a}+N\cdot(1,\ldots,1)$. Once again relying on the choice of $N$, as well as the classification of projective modules as being semi-induced, we conclude that these extension groups must be zero, as desired.\\
\end{proof}

These two propositions are the main ingredients in the following classification theorem.\\

\begin{theorem}\label{acyclicclass}
Let $V$ be a finitely generated $\FI^m$-module. Then $V$ is acyclic with respect to the positive local cohomology functors if and only if it falls into an exact sequence of the form
\[
0 \rightarrow V_T \rightarrow V \rightarrow V_F \rightarrow 0
\]
where $V_F$ is semi-induced, and $V_T$ is $B$-torsion.\\
\end{theorem}

\begin{proof}
The two previous propositions imply the backward direction. Otherwise, $V$ can always be fit into an exact sequence
\[
0 \rightarrow V_T \rightarrow V \rightarrow V_F \rightarrow 0
\]
where $V_T$ is $B$-torsion and $V_F$ is $B$-torsion-free. It therefore suffices to prove the following claim: If $V$ is a finitely generated $\FI^m$-module for which $H^i_B(V) = 0$ for \textbf{all} $i \geq 0$, then $V$ is semi-induced.\\

Indeed, using the previous propositions it isn't hard to see that in this case the complex
\[
\filcom V: 0 \rightarrow V \rightarrow F^{(0)} \rightarrow \ldots \rightarrow F^{(l)} \rightarrow 0
\]
of Theorem \ref{cechcomplex} is exact. Applying the homology functor and using Theorem \ref{inducedproperties}, we may therefore conclude that $V$ is acyclic with respect to the homology functors and is therefore semi-induced by Theorem \ref{inducedproperties} as well.\\
\end{proof}

\subsection{Consequences of the acyclic classification}

In this section we list consequences of the classification of local cohomology acyclic modules. The first consequence was already hinted at in the proof of the classification.

\begin{theorem}\label{cechcomplexlc}
Let $V$ be a finitely generated $\FI^m$-module. Then for each $i \geq 0$, there are isomorphisms
\[
H^{i-1}(\filcom V) \cong H^i_B(V)
\]
\end{theorem}

\begin{proof}
We have an exact sequence
\[
0 \rightarrow V_T \rightarrow V \rightarrow V_F \rightarrow 0
\]
where $V_T$ is $B$-torsion and $V_F$ is $B$-torsion-free. Thus $H^i_B(V) = H^i_B(V_F)$ for all $i \geq 1$. The module $V_F$ admits an injection
\[
0\rightarrow V_F \rightarrow \Sigma_{\mathbf{a}} V_F = F^{(0)} \rightarrow Q^{(0)} \rightarrow 0
\]
for some $\mathbf{a} \in \N^m$. We have that $H^0(\filcom V) = H^0_B(Q^{(0)}) = H^1_B(V_F) = H^1_B(V)$. Repeating this argument for $Q^{(0)}$, we conclude the desired isomorphisms for each $i$.\\
\end{proof}

As an immediate consequence for this theorem, we will be able to prove a variety of finiteness statements.\\

\begin{corollary}\label{finite}
Let $V$ be a finitely generated $\FI^m$-module. Then
\begin{enumerate}
\item $H^i_B(V)$ is finitely generated for all $i \geq 0$;
\item $H^i_B(V) = 0$ for all $i \gg 0$;
\end{enumerate}
\end{corollary}

\begin{proof}
Both statements follow from Theorem \ref{cechcomplexlc} and the Noetherian property.\\
\end{proof}

\section{Local cohomology and regularity}

\subsection{Definitions}

In this section we consider more deeply the connection between homology and local cohomology. The overall arc of these sections, as well as many of the definitions, is designed to invoke the classical theory. Namely, we will be following the work of Maclagan and Smith \cite{MS}.\\

\begin{definition}
Let $V$ be a finitely generated $\FI^m$-module. The \textbf{Castelnuovo-Mumford regularity}, or \textbf{CM regularity}, is the set of integral vectors $\mathbf{r} \in \Zn^m$ such that:
\begin{enumerate}
\item for $i \geq 0$, $H^i_B(V)_{\mathbf{n}} = 0$ whenever $\mathbf{n} \in \cup_{\mathbf{a} \in \N^m, |\mathbf{a}|=i} \cup_j (\mathbf{r} - \mathbf{a} +e_j + \N^m)$;
\item $\mathbf{r} \in \HD_0(V)$.\\
\end{enumerate}
We denote the CM regularity of $V$ by $\CMreg(V)$. It will be convenient for us to also set $\CMreg_+(V)$ to be the collection of vectors which satisfy the first of the two conditions above. Namely,
\[
\CMreg(V) = \CMreg_+(V) \cap \HD_0(V)
\]
\end{definition}

\begin{remark}
In the classical setting, the second requirement on $\mathbf{r}$ is not necessary. In our setting, however, the homology acyclic modules - i.e. the semi-induced modules - are also local cohomology acyclic (see Theorem \ref{acyclicclass}). As a consequence, the local cohomology modules cannot detect $\HD_0(V)$. What we will find, however, is that they can detect pieces of $\HD_i(V)$ for higher $i$ (see Theorem \ref{fullversion}).\\
\end{remark}

It will be illustrative for us to take a moment and look at an example.\\

\begin{example}
Let $V$ be that $\FI^2$-module given by
\[
V_{(x,y)} =  \begin{cases} k &\text{if $y = 0$}\\ 0 &\text{otherwise.}\end{cases}
\]
whose transition maps are the identity on the $x$-axis, and the zero map elsewhere. Then $V$ is a $B$-torsion module, and subsequently has trivial higher local cohomology groups by Proposition \ref{tortrivial}. It follows that,
\[
\CMreg(V) \cap \N^2 = \{\mathbf{r} \mid V_\mathbf{n} = 0 \text{ if } \mathbf{n} \in \cup_j \mathbf{r} + e_j + \N^2\} \cap \HD_0(V) = \{\mathbf{r} \mid V_\mathbf{n} = 0 \text{ if } \mathbf{n} \in \cup_j \mathbf{r} + e_j + \N^2\} = \{(0,1) + \N^2\}.
\]
The module $V$ admits a presentation
\[
0 \rightarrow K \rightarrow M(0,0) \rightarrow V \rightarrow 0
\]
In other words, $K$ is the module defined on points by
\[
K_{(x,y)} = \begin{cases} k &\text{if $y \neq 0$}\\ 0 &\text{otherwise.}\end{cases}
\]
Applying the local cohomology functor to this presentation, and using work of previous sections, we conclude that $H^{i}_B(V) = H^{i+1}_B(K)$ for all $i$. Therefore,
\[
\CMreg(K) \cap \N^2 = \{\mathbf{r} \mid V_{\mathbf{n}} = 0 \text{ if } \mathbf{n} \in \cup_{j,s} (\mathbf{r} - e_j + e_s +\N^2)\} \cap \HD_0(K) \cap \N^2
\]
It is not hard to see that $\HD_0(K) \cap \N^2 = \{(0,1) + \N^2\}$, and that
\[
\{\mathbf{r} \mid V_{\mathbf{n}} = 0 \text{ if } \mathbf{n} \in \cup_{j,s} (\mathbf{r} - e_j + e_s +\N^2)\} \cap \N^2 = \{(0,2) + \N^2\}.
\]
Thus,
\[
\CMreg(K) \cap \N^2 = \{(0,2) + \N^2\}.
\]
\end{example}

\subsection{Theorem \ref{reglc} for $B$ torsion modules}
Our first goal will be to show Theorem \ref{reglc} for $B$-torsion modules. In this setting, the statement of this theorem reduces to the following. Note that facts which arise in the proof of this theorem will effectively form the base case of an eventual inductive proof of Theorem \ref{reglc}.\\

\begin{theorem} \label{torversion}
Let $V$ be a finitely generated $B$-torsion module. If $\mathbf{c}$ is any integral vector with positive coordinates, and $\mathbf{r}$ is such that $V_{\mathbf{n}} = 0$ whenever $\mathbf{n} \in \cup_j \mathbf{r} + e_j + \N^m$, then there exists a complex
\[
\ldots \to F^{(1)} \rightarrow F^{(0)} \stackrel{\partial}\rightarrow V \rightarrow 0
\]
such that
\begin{enumerate}
\item $F^{(i)}$ is a semi-induced module generated in degree $\mathbf{r} + i\cdot\mathbf{c}$;
\item the homology of the complex is $B$-torsion.\\
\end{enumerate}
\end{theorem}

To prove this theorem, we will need an alternative characterization of homology. The following definition is motivated by work of Church and Ellenberg, as well as work of Gan and the first author \cite{GL2}. It is perhaps most natural to describe this construction using a different version of $\FI^m$. Note that $\FI^m$ is equivalent to the category whose objects are $m$-fold products of finite sets
\[
S_1 \times \ldots \times S_m
\]
and whose morphisms are $m$-fold products of injections.\\

\begin{definition}
For a given non-negative integer $a$, and an $\FI^m$-module $V$, we define a new $\FI^m$-module
\[
(\Sigma_{-a}V)_{S_1 \times \ldots \times S_m} = \oplus_{T_j \subseteq S_j, \sum_j |T_j| = a} V(S_1 \setminus T_1 \times \ldots \times S_m \setminus T_m) \otimes_k \epsilon_{T_1 \times \ldots \times T_m}
\]
where $\epsilon_{T_1 \times \ldots \times T_m} = \epsilon_{T_1} \boxtimes \ldots \boxtimes \epsilon_{T_m}$ and $\epsilon_T$ is the sign representation of $\Sn_T$. Given a $m$-fold product of injections $f_1 \times \ldots f_m:S_1 \times \ldots \times S_m \rightarrow S'_1 \times \ldots \times S'_m$ we may define the induced map on $\Sigma_{-a}V$ by
\begin{eqnarray*}
&(\Sigma_{-a}V)(f_1\times\ldots\times f_m) (v \otimes (t_{1,1} \wedge \ldots \wedge t_{1,|T_1|}) \otimes \ldots \otimes (t_{m,1} \wedge \ldots \wedge t_{m,|T_m|})) =\\
&V((f_1 \times \ldots \times f_m)|_{S_1 \setminus T_1 \times \ldots \times S_m \setminus T_m})(v) \otimes (f_1(t_{1,1}) \wedge \ldots \wedge f_1(t_{1,|T_1|})) \otimes \ldots \otimes (f_m(t_{m,1}) \wedge \ldots \wedge f_m(t_{m,|T_m|})),
\end{eqnarray*}
where $T_j = \{t_{j,1},\ldots,t_{j,|T_j|}\} \subseteq S_j$ for each $j$, and $v \in V(S_1 \setminus T_1 \times \ldots \times S_m \setminus T_m)$.\\

Next we claim that the collection of $S_{-a}V$, with $a$ varying among positive integers, form an complex of $\FI^m$-modules. Indeed, we can define a map
\[
V(S_1 \setminus T_1 \times \ldots \times S_m \setminus T_m) \otimes_k \epsilon_{T_1 \times \ldots \times T_m} \rightarrow S_{-(a-1)}V(S_1 \times \ldots \times S_m)
\]
for any given $T_1 \times \ldots \times T_m \subseteq S_1 \times \ldots \times S_m$, with $\sum_j |T_j| = a$ and a chosen labeling $T_1 = \{t_{1,1},\ldots T_{1,|T_1|}\}$, $\ldots$, $T_m = \{t_{m,a-|T_m| + 1},\ldots, t_{m,a}\}$ by setting
\begin{align*}
& v \otimes (t_{1,1} \wedge \ldots \wedge t_{1,|T_1|}) \otimes \ldots \otimes (t_{m,a-|T_m|+1} \wedge \ldots \wedge t_{m,|T_m|}) \mapsto\\
& \sum_j (-1)^j\iota_j(v) \otimes (t_{1,1} \wedge \ldots \wedge t_{1,|T_1|}) \otimes \ldots \otimes (t_{m,a-|T_m|+1} \wedge \ldots \wedge t_{m,|T_m|})
\end{align*}
where the wedge products in the $j$-th summand of the right hand side are missing the term $t_{l,j}$, and $\iota_j$ is the map which adds $t_{l,j}$ to $S_l$. In other words, the differentials in the complex $S_{-\bullet}V$ are defined by all possible ways to add an element from the $T_j$ back into the $S_j$.\\

The complex $S_{-\bullet}V$ is known as the \textbf{Koszul complex} associated to $V$.
\end{definition}

One key fact about the Koszul complex is that is actually computes the homology of $V$. The following theorem was proven independently in \cite{CE} and \cite{GL2} for $\FI$-modules. The proof in our setting is similar, and we omit it.\\

\begin{theorem}\label{koszulhom}
If $V$ is an $\FI$-module, then for all $i \geq 0$
\[
H_i(V) \cong H_i(S_{-\bullet}V).
\]
\end{theorem}

This theorem is the major piece needed to prove Theorem \ref{torversion}. Recall that if $V$ is an $\FI^m$-module, and $\mathbf{n} \in \N^m$, then $V|_{\geq \mathbf{n}}$ is the submodule of $V$ generated by $V_{\mathbf{n}}$.\\

\begin{proof}[Proof of Theorem \ref{torversion}]
Let $V$ be a finitely generated $B$-torsion $\FI^m$-module, let $\mathbf{r} \in \CMreg(V)$ and $\mathbf{c}$ be any integral vector with positive coordinates. By definition of $V|_{\geq \mathbf{r}}$, there is an exact sequence
\begin{eqnarray}
0 \rightarrow K^{(1)} \rightarrow M(V_{\mathbf{r}}) \rightarrow V|_{\geq \mathbf{r}} \rightarrow 0 \label{firstsyz}
\end{eqnarray}
Set $F^{(0)} = M(V_{\mathbf{r}})$.\\

We claim that $\mathbf{r} + \mathbf{c} \in \HD_0(K^{(1)})$. Indeed, the exact sequence (\ref{firstsyz}) implies that $H_1(V|_{\geq \mathbf{r}}) = H_0(K^{(1)})$. Using Theorem \ref{koszulhom}, one sees that $H_1(V|_{\geq \mathbf{r}})_{\mathbf{r}+\mathbf{c}} = 0$ as it is a subquotient of a direct sum of modules of the form
\[
(V|_{\geq \mathbf{r}})([r_1+c_1] - S_1 \times \ldots \times [r_m+c_m]-S_m)
\]
where $\sum_j |S_j| = 1$ and $S_j \subseteq [r_j+c_j]$. Because $c_j$ is positive for all $j$, $r_j+c_j - |S_j| \geq r_j$ with this inequality being strict for at least one choice of $j$. By the definition of $\CMreg(V)$ it follows that
\[
(V|_{\geq \mathbf{r}})([r_1+c_1] - S_1 \times \ldots \times [r_m+c_m]-S_m) = 0
\]
as desired. As a consequence, $\mathbf{r} + \mathbf{c} \in \HD_0(K^{(1)})$ and the quotient $K^{(1)} / (K^{(1)}|_{\geq \mathbf{r} + \mathbf{c}}$) is $B$-torsion. Set $F^{(1)} = M(K^{(1)}_{\mathbf{r} + \mathbf{c}})$.\\

Proceeding inductively, assume that we have constructed an exact sequence
\[
0 \rightarrow K^{(i+1)} \rightarrow F^{(i)} \rightarrow K^{(i)}|_{\geq \mathbf{r} + i\cdot \mathbf{c}} \rightarrow 0
\]
where $\mathbf{r} + i\cdot \mathbf{c} \in \HD_0(K^{(i)})$, and $F^{(i)} = M(K^{(i)}_{\mathbf{r} + i\cdot \mathbf{c}})$. By construction we know that
\[
H_1(K^{(i)}|_{\geq \mathbf{r} + i\cdot \mathbf{c}}) \cong H_0(K^{(i+1)}).
\]
Let $\mathbf{n} \in \N^m$ be such that $\mathbf{n} - (\mathbf{r} + (i+1)\cdot \mathbf{c}) \in \N^m$ is non-zero. We have to show
\[
H_1(K^{(i)}|_{\geq \mathbf{r} + i\cdot \mathbf{c}})_{\mathbf{n}} = 0
\]
It follows from Theorem \ref{koszulhom}, the definition of $\mathbf{n}$, the fact that $\mathbf{r} + i\cdot \mathbf{c} \in \HD_0(K^{(i)})$, and that
\[
H_1(K^{(i)}|_{\geq \mathbf{r} + i\cdot \mathbf{c}})_{\mathbf{n}} \cong H_1(K^{(i)})_{\mathbf{n}}.
\]

By induction we know that
\[
H_1(K^{(i)})_{\mathbf{n}} \cong H_2(K^{(i-1)}|_{\geq \mathbf{r}+(i-1)\cdot \mathbf{c}})_{\mathbf{n}}
\]
and using the same argument as above we conclude $H_1(K^{(i)})_{\mathbf{n}} \cong H_2(K^{(i-1)})_{\mathbf{n}}$. Inductively we eventually conclude that
\[
H_1(K^{(i)})_\mathbf{n} \cong H_{i+1}(V|_{\geq \mathbf{r}})_{\mathbf{n}} \cong H_{i+1}(V)_{\mathbf{n}} = 0,
\]
once again using our choice of $\mathbf{n}$. This concludes the proof.\\
\end{proof}

\subsection{The proof of Theorem \ref{reglc}}

We begin by stating a theorem which implies Theorem \ref{reglc}.\\

\begin{theorem}\label{fullversion}
Let $V$ be a finitely generated $\FI^m$-module, $i\geq 1$, and let $\mathbf{c}_i \in \Zn^m$ be any vector such that $\mathbf{c}_i - \mathbf{a} \in \CMreg_+(V)$ for all vectors $\mathbf{a} \in \N^m$ such that $|\mathbf{a}| = i$. Then one has,
\[
\mathbf{c}_i \in \HD_i(V)
\]
\end{theorem}

To see why this theorem implies Theorem \ref{reglc}, one looks towards the proof of Theorem \ref{torversion}. Indeed, one can see that the key ingredient that makes the induction work in that argument is the fact that
\[
H_1(K^{(i)})_\mathbf{n} \cong H_{i+1}(V|_{\geq \mathbf{r}})_{\mathbf{n}} \cong H_{i+1}(V)_{\mathbf{n}} = 0
\]
whenever $\mathbf{n}$ is at least one positive step from $\mathbf{r} + (i+1) \cdot \mathbf{c}$. It follows that the proof of \ref{torversion} can be replicated to prove Theorem \ref{reglc} so long as we have Theorem \ref{fullversion}.\\

\begin{proof}[Proof of Theorem \ref{fullversion}]
Fix $\mathbf{c}_i \in \N^m$ as in the statement of the theorem. We will proceed by induction on the largest $d$ such that $H^d_B(V) \neq 0$. This quantity is finite by Corollary \ref{finite}. Note that if this quantity is 0, then the theorem follows by the proof of Theorem \ref{torversion} and Theorem \ref{acyclicclass}. Otherwise assume that $V$ is not local cohomology acyclic and recall there is an exact sequence
\begin{eqnarray*}
0 \rightarrow V_T \rightarrow V \rightarrow V_F \rightarrow 0\label{exactsequence}
\end{eqnarray*}
where $V_T$ is $B$-torsion, and $V_F$ is $B$-torsion-free. This yields exact sequences for all $s \geq 0$
\[
H_s(V_T) \rightarrow H_s(V) \rightarrow H_s(V_F).
\]
It follows that $\mathbf{c}_i \in \HD_i(V)$ so long as $\mathbf{c}_i \in \HD_i(V_T) \cap \HD_i(V_F)$. On the other hand, we have that $V_T = H^0_B(V)$ and therefore $\mathbf{c}_i \in \HD_i(V_T)$ by definition of $\mathbf{c}_i$ as well as Theorem \ref{koszulhom}. It remains to argue that $\mathbf{c}_i \in \HD_i(V_F)$.\\

The exact sequence (\ref{exactsequence}) implies that
\[
H^s_B(V) = H^s_B(V_F)
\]
whenever $s \geq 1$. Note that this also implies that
\[
\CMreg_+(V) \subseteq \CMreg_+(V_F).
\]
Theorem \ref{shiftthm} tells us that there is an exact sequence
\begin{eqnarray*}
0 \rightarrow V_F \rightarrow F \rightarrow Q \rightarrow 0 \label{exactsequence2}
\end{eqnarray*}
Where $F$ is semi-induced. An induction implies that $\mathbf{c}'_i \in \HD_i(Q)$ whenever $\mathbf{c}'_i$ satisfies that $\mathbf{c}'_i - \mathbf{a} \in \CMreg(Q)$ for all $\mathbf{a} \in \N^m$ with $|\mathbf{a}| = i$. However, (\ref{exactsequence2}) tells us that $H^s_B(Q) \cong H^{s+1}_B(V_F)$ for $s \geq 0$ which implies that $\CMreg_+(V_F) \subseteq \cap_j (\CMreg_+(Q) + e_j) $. Once again we may leverage (\ref{exactsequence2}) to find
\[
H_i(V_F) \cong H_{i+1}(Q),
\]
for $i \geq 1$.\\

Let $\mathbf{a} \in \N^m$ satisfies $|\mathbf{a}| = i+1$. Then we can find some index $j$ such that $a_j \neq 0$. Writing $\mathbf{a} = \mathbf{a}'+e_j$ we have,
\[
\mathbf{c}_i - \mathbf{a} = (\mathbf{c}_i - \mathbf{a}') - e_j \in \CMreg_+(V) - e_j \subseteq \CMreg_+(V_F) - e_j \subseteq \CMreg_+(Q),
\]
by definition of $\mathbf{c}_i$. By induction we conclude that $\mathbf{c}_i \in \HD_{i+1}(Q)$ and so $\mathbf{c}_i \in \HD_i(V_F)$. This concludes the proof.\\
\end{proof}

\begin{remark}\label{technicalissue}
It was critical in the above proof that $\CMreg_+(V_F) \subseteq \cap_j (\CMreg_+(Q) + e_j)$. For this statement to be true, it is necessary for the CM regularity include vectors in $\Zn^m - \N^m$.\\
\end{remark}

\end{document}